\documentclass{aic}

\usepackage{amsfonts,amsmath,latexsym,color,epsfig,hyperref}

\newtheorem{theorem}{Theorem}
\newtheorem{lemma}{Lemma}

\aicAUTHORdetails{%
  title = {The Ramsey number of books}, %% please capitalize all significant words
  author = {David Conlon},
  plaintextauthor = {David Conlon},
  plaintexttitle = {The Ramsey number of books}, %%  title without math or LaTeX
  runningtitle = {The Ramsey number of books}, 
  runningauthor = {David Conlon},
  copyrightauthor = {D. Conlon},
  keywords = {Ramsey numbers, regularity method, Ramsey multiplicity},
}   %%% END \aicAUTHORdetails

\aicEDITORdetails{%
   year={2019},
   number={3},
   received={9 August 2018},   % received date: example: 7 January 2017
   published={30 October 2019},  % published date
   doi={10.19086/aic.10808},      % XXX = number of paper, e.g. aic006 for paper#6
}   %%% END \aicEDITORdetails

\begin{document}

\begin{frontmatter}[classification=text]
%% EDITOR: this will force the keywords to appear right after the Abstract.
%%   If the abstract is too long and would force the keywords off the
%%   front page, please comment out % [classification=text] above
%%   This way the keywords will be floated on the bottom of the first page
%%   even though the Abstract spills over to the next page.

%%% AUTHOR: Title goes here.  This line is optional.  You must use it
%%   if title has footnote attached or requires nontrivial typesetting,
%%   e.g., inclusion of linebreaks to force nice layout.
\title{The Ramsey number of books} %% please capitalize all significant words

%%% AUTHOR:
%%% List all authors. If you wish, place grant acknowledgements in \thanks.
%%% In brackets include a short tag for each author.
\author[dcon]{David Conlon\thanks{Supported by a Royal Society University Research Fellowship and ERC Starting Grant 676632.}}

%%% AUTHOR: Abstract goes here
\begin{abstract}
We show that in every two-colouring of the edges of the complete graph $K_N$ there is a monochromatic $K_k$ which can be extended in at least $(1 + o_k(1))2^{-k}N$ ways to a monochromatic $K_{k+1}$. This result is asymptotically best possible, as may be seen by considering a random colouring. Equivalently, defining the book $B_n^{(k)}$ to be the graph consisting of $n$ copies of $K_{k+1}$ all sharing a common $K_k$, we show that the Ramsey number $r(B_n^{(k)}) = 2^k n + o_k(n)$. In this form, our result answers a question of Erd\H{o}s, Faudree, Rousseau and Schelp and establishes an asymptotic version of a conjecture of Thomason.
\end{abstract}
\end{frontmatter}

%%% AUTHOR: body of paper starts here
\section{Introduction}

The Ramsey number $r(H)$ of a graph $H$ is the smallest natural number $N$ such that every two-colouring of the edges of the complete graph $K_N$ contains a monochromatic copy of $H$. The problem of determining Ramsey numbers is notoriously hard. For instance, when $H$ is a complete graph, work of Erd\H{o}s and Szekeres~\cite{E47, ESz35} in the 1930s and 40s showed that $\sqrt{2}^t \leq r(K_t) \leq 4^t$, but the only improvements to these bounds since that time~\cite{C09, S75} have been to lower order terms.

We investigate the Ramsey numbers of books, a study which bears close relation to the problem of determining $r(K_t)$. The book $B_n^{(k)}$ is the graph consisting of $n$ copies of $K_{k+1}$, all sharing a common $K_k$. Embracing the metaphor, we refer to the common $K_k$ as the spine of the book and the $n$ points completing each copy of $K_{k+1}$ as the pages or leaves.

The Ramsey problem for these books was first studied by Erd\H{o}s, Faudree, Rousseau and Schelp~\cite{EFRS78} and then by Thomason~\cite{T82}. Both papers contain bounds of the form
\[2^k n + o_k(n) \leq r(B_n^{(k)}) \leq 4^k n,\]
where the lower bound follows from considering the random graph $G(n,1/2)$ and the upper bound from a standard neighbourhood chasing argument. 
In their paper, Erd\H{o}s et al.~asked whether one of these bounds might be asymptotically correct and Thomason conjectured that the lower bound is. In fact, he made a very precise conjecture about the value of $r(B_n^{(k)})$, namely, that
\[r(B_n^{(k)}) \leq 2^k(n + k - 2) + 2.\]
For $k = 2$, this conjecture is known to hold~\cite{RS78} and is tight for infinitely many values of $n$. The main contribution of this paper is a proof of an approximate version of Thomason's conjecture, thus answering the question of Erd\H{o}s et al.~(see also~\cite{NRS05}). 

\begin{theorem} \label{thm:main}
For every natural number $k$,
\[r(B_n^{(k)}) = 2^k n + o_k(n).\]
\end{theorem}

To see something of why this is interesting, suppose that we have a red/blue-colouring of $K_N$ with no monochromatic copy of $K_t$. Then this colouring contains no monochromatic book $B_n^{(k)}$ with $n$ equal to the off-diagonal Ramsey number $r(K_{t -k}, K_t)$. Indeed, suppose that the book is red. If the set induced by the pages contains a blue $K_t$, we have a contradiction, so it must contain a red $K_{t-k}$. But together with the red spine $K_k$, this forms a red $K_t$. Hence, by Theorem~\ref{thm:main}, if $t$, and therefore $n$, is sufficiently large in terms of $k$, we have
\[r(K_t) \leq r(B_n^{(k)}) \leq 2^{k+1} r(K_{t-k}, K_t) \leq 2^{k+1} \binom{2t - k}{t-k},\]
where the last inequality follows from a classical estimate of Erd\H{o}s and Szekeres~\cite{ESz35}.
In particular, if the theorem applied for $t$ linear in $k$, this would give an exponential improvement on the upper bound for diagonal Ramsey numbers. Unfortunately, our Theorem~\ref{thm:main} is very far from achieving this goal, since in order to obtain an error term of the form $\epsilon n$ we require $n$ to be at least a tower of twos whose height is a function of $k$ and $1/\epsilon$.

Another motivation for Theorem~\ref{thm:main} is its relation to a well-known, but false, conjecture of Erd\H{o}s~\cite{E62} (see also~\cite{BR80}) asserting that every two-colouring of the edges of $K_N$ contains at least 
$$(1 + o_k(1)) 2^{1-\binom{k}{2}} \binom{N}{k}$$ 
monochromatic copies of $K_k$. That is, he conjectured that a random colouring should asymptotically minimise the number of monochromatic copies of $K_k$. While true for $k = 3$ by a result of Goodman~\cite{G59}, this conjecture was disproved for $k \geq 4$ by Thomason~\cite{T89}. However, Theorem~\ref{thm:main} is equivalent to a local version of Erd\H{o}s' conjecture, saying that there is some monochromatic copy of $K_{k-1}$ which is contained in asymptotically as many monochromatic $K_k$ as in a random colouring. In some ways, this interpretation is more appealing than the original formulation in terms of books. It also connects our result with the study of Ramsey multiplicity, which has drawn considerable attention in recent years (see, for instance,~\cite{C12, CKPSTY13, F08, HHKNR12, JST96}).

\section{Preliminaries}

In this section, we collect several results that we will need for the proof of Theorem~\ref{thm:main}.

\subsection{Regularity and counting lemmas}

One of the main ingredients in our proof is a simple corollary of Szemer\'edi's regularity lemma. To state this fundamental result, we first recall some standard definitions. Given two vertex sets $U$ and $V$ in a graph, the density $d(U,V)$ between them is given by $d(U,V) = e(U,V)/|U||V|$. A bipartite graph between two vertex sets $U$ and $V$ is said to be $\epsilon$-regular if, for all sets $U' \subseteq U$, $V' \subseteq V$ with $|U'| \geq \epsilon |U|$ and $|V'| \geq \epsilon |V|$, $|d(U', V') - d(U,V)| \leq \epsilon$. A partition $V(G) = \cup_{i=1}^m V_i$ of the vertex set of a graph $G$ is said to be equitable if $||V_i| - |V_j|| \leq 1$ for all $i$ and $j$. The regularity lemma is now as follows.

\begin{lemma} \label{lem:reg}
For every $0 < \epsilon < 1$ and natural number $m_0$, there exists a natural number $M$ such that every graph $G$ with at least $m_0$ vertices has an equitable partition $V(G) = \cup_{i=1}^m V_i$ with $m_0 \leq m \leq M$ parts such that all but $\epsilon m^2$ pairs $(V_i, V_j)$ with $1 \leq i \neq j \leq m$ are $\epsilon$-regular.
\end{lemma}

We will also need the following lemma from~\cite{CF12}. We say that a subset $U$ of the vertex set of a graph $G$ is $\epsilon$-regular if the pair $(U, U)$ is $\epsilon$-regular.

\begin{lemma} \label{lem:regpiece}
For every $0 < \epsilon < 1$, there exists a constant $\delta$ such that every graph $G$ contains an $\epsilon$-regular vertex subset $U$ with $|U| \geq \delta |V(G)|$. 
\end{lemma}

The key lemma we will need is the following. We note a superficial similarity to a lemma used in the proof of the induced removal lemma~\cite{AFKS00}, though that lemma requires the stronger condition that every pair $(W_i, W_j)$ be regular.

\begin{lemma} \label{lem:main}
For every $0 < \eta < 1$ and natural number $m_0$, there exists a natural number $M$ such that every graph $G$ with at least $m_0$ vertices has an equitable partition $V(G) = \cup_{i=1}^m V_i$ with $m_0 \leq m \leq M$ parts and subsets $W_i \subseteq V_i$ such that $W_i$ is $\eta$-regular for all $i$ and, for all but $\eta m^2$ pairs $(i, j)$ with $1 \leq i \neq j \leq m$, $(V_i, V_j)$, $(W_i, V_j)$ and $(W_i, W_j)$ are $\eta$-regular with $|d(W_i, V_j) - d(V_i, V_j)| \leq \eta$ and $|d(W_i, W_j) - d(V_i, V_j)| \leq \eta$.
\end{lemma}

\begin{proof}
Apply the regularity lemma, Lemma~\ref{lem:reg}, to $G$ with $\epsilon = \eta \cdot \delta(\eta)$, with $\delta$ as in Lemma~\ref{lem:regpiece}. This yields an equitable partition $V(G) = \cup_{i=1}^m V_i$ where all but $\epsilon m^2 \leq \eta m^2$ pairs $(V_i, V_j)$ with $1 \leq i \neq j \leq m$ are $\epsilon$-regular. Within each piece $V_i$, now apply Lemma~\ref{lem:regpiece} to find a set $W_i$ of order at least $\delta(\eta)$ which is $\eta$-regular. Note that if $(V_i, V_j)$ is $\epsilon$-regular, then, since $|W_i| \geq \delta |V_i|$ and $\epsilon = \eta \cdot \delta(\eta)$, the pairs $(W_i, V_j)$ and $(W_i, W_j)$ are $\eta$-regular with $|d(W_i, V_j) - d(V_i, V_j)| \leq \epsilon \leq \eta$  and $|d(W_i, W_j) - d(V_i, V_j)| \leq \eta$.
\end{proof}

In order to apply Lemma~\ref{lem:main}, we need a standard counting lemma (see, for example,~\cite[Theorem~18]{RS10}). We use the shorthand $x \pm \delta$ to indicate a quantity that lies between $x - \delta$ and $x + \delta$.

\begin{lemma} \label{lem:count}
For any $\delta > 0$ and any natural number $k$, there is $\eta > 0$ such that if $U_1, \dots, U_k$ are (not necessarily distinct) vertex sets with $(U_i, U_{i'})$ $\eta$-regular of density $d_{i,i'}$ for all $1 \leq i < i' \leq k$, then there are
\[\prod_{i < i'} d_{i,i'} \prod_{i=1}^k |U_i| \pm \delta \prod_{i=1}^k |U_i|\]
copies of $K_k$ with vertex $i$ in $U_i$ for each $1 \leq i \leq k$.
\end{lemma}

In practice, we will always use this lemma in the following form.

\begin{lemma} \label{lem:count2}
For any $\delta > 0$ and any natural number $k$, there is $\eta > 0$ such that if $U_1, \dots, U_k$, $U_{k+1}, \dots, U_{k+\ell}$ are (not necessarily distinct) vertex sets with $(U_i, U_{i'})$ $\eta$-regular of density $d_{i,i'}$ for all $1 \leq i < i' \leq k$ and $1 \leq i \leq k < i' \leq k + \ell$ and $d_{i,i'} \geq \delta$ for all $1 \leq i < i' \leq k$, then there is a copy of $K_k$ with vertex $i$ in $U_i$ for each $1 \leq i \leq k$ which is contained in at least
\[\sum_{j=1}^\ell \left(\prod_{i=1}^k d_{i, k+j} - \delta\right) |U_{k+j}|\]
copies of $K_{k+1}$ with vertex $k+1$ in $\cup_{j=1}^\ell U_{k+j}$.
\end{lemma}

\begin{proof}
By Lemma~\ref{lem:count} applied with $\delta' = \delta^{\binom{k}{2} + 1}/2$ instead of $\delta$, there exists $\eta_0 = \eta(\delta', k)$ such that the number of copies of $K_k$ with vertex $i$ in $U_i$ for each $1 \leq i \leq k$ is at most
\[\prod_{1 \leq i < i' \leq k} d_{i,i'} \prod_{i=1}^k |U_i| + \delta' \prod_{i=1}^k |U_i|.\]
Moreover, by repeated application of Lemma~\ref{lem:count} with $k+1$ parts, there exists $\eta_1 = \eta(\delta', k+1)$ such that the number of copies of $K_{k+1}$ with vertex $i$ in $U_i$ for each $1 \leq i \leq k$ and vertex $k+1$ in $\cup_{j=1}^\ell U_{k+j}$ is at least
\[\sum_{j=1}^\ell \prod_{i=1}^k d_{i, k+ j} |U_{k+j}| \prod_{1 \leq i < i' \leq k} d_{i,i'}  \prod_{i=1}^{k} |U_i| - \delta' \sum_{j=1}^\ell |U_{k+j}| \prod_{i=1}^{k} |U_i|.\]
Therefore, for $\eta =\min(\eta_0, \eta_1)$, there must be some $K_k$ which is in at least
\begin{align*}
\frac{ \sum_{j=1}^\ell \prod_{i=1}^k d_{i, k+ j} |U_{k+j}| \prod_{1 \leq i < i' \leq k} d_{i,i'} - \delta' \sum_{j=1}^\ell |U_{k+j}|}{\prod_{1 \leq i < i' \leq k} d_{i,i'} + \delta'} & \geq \frac{\sum_{j=1}^\ell \left(\prod_{i=1}^k d_{i, k+j} - \delta^{-\binom{k}{2}} \delta'\right) |U_{k+j}|}{1 + \delta^{-\binom{k}{2}} \delta'}\\
& \geq \sum_{j=1}^\ell \left(\prod_{i=1}^k d_{i, k+j} - \delta\right) |U_{k+j}|
\end{align*}
copies of $K_{k+1}$, as required.
\end{proof}

\subsection{A coloured extremal result}

We will need a coloured version of the celebrated Erd\H{o}s--Stone theorem~\cite{ES46}. Recall that a blow-up of a graph $H$ is a graph where each vertex of $H$ is replaced by a vertex set and the bipartite graph between two such vertex sets is complete whenever the corresponding vertices are joined by an edge.

\begin{lemma} \label{lem:ES1}
For any natural numbers $k$ and $t$ and any $\delta > 0$, there exists a natural number $n_0$ such that if the edges of the complete graph on $n \geq n_0$ vertices are coloured in red and blue, then, provided the blue density is at least $1 - \frac{1}{k-1} + \delta$, there is a blue blow-up of $K_k$ with $t$ vertices in each part, where each part is itself a monochromatic clique.
\end{lemma}

\begin{proof}
Since the blue density is at least $1 - \frac{1}{k-1} + \delta$, the Erd\H{o}s--Stone theorem implies that for $n \geq n_0$ there is a blue blow-up of $K_k$ with at least $r(K_t)$ vertices in each part. Applying Ramsey's theorem within each part then gives the required monochromatic cliques.
\end{proof}

In practice, we will use a slight variant of this lemma, where the underlying graph is not necessarily complete.

\begin{lemma} \label{lem:ES2}
For any natural numbers $k$ and $t$ and any $\delta > 0$, there exists a natural number $n_1$ and $\epsilon > 0$ such that if the edges of a graph on $n \geq n_1$ vertices with $(1 - \epsilon) \binom{n}{2}$ edges are coloured in red and blue, then, provided the blue density is at least $1 - \frac{1}{k-1} + \delta$, there is a blue blow-up of $K_k$ with $t$ vertices in each part, where each part is itself a monochromatic clique.
\end{lemma}

\begin{proof}
Suppose first that the $\epsilon \binom{n}{2}$ missing edges are coloured blue, so that the underlying graph is complete. Then, by Lemma~\ref{lem:ES1}, every subset of the graph of order $n_0$ contains the required blow-up of $K_k$ with monochromatic parts of order $t$. But then, for $n$ sufficiently large in terms of $n_0$, there must be at least
\[\binom{n}{n_0}/\binom{n - kt}{n_0 - kt} = \binom{n}{kt}/\binom{n_0}{kt} \geq \frac{n^{kt}}{2 n_0^{kt}}\]
such blow-ups of $K_k$. However, at most $\epsilon n^2 \cdot n^{kt-2} = \epsilon n^{kt}$ such copies contain an edge from the missing set. Therefore, for $\epsilon < 1/2 n_0^{kt}$, we must have the required blue blow-up of $K_k$ with at least $t$ vertices in each part, where each part is a monochromatic clique. 
\end{proof}

\subsection{Some technical lemmas}

The proof requires a small degree of optimisation, almost all of which is contained in the following two lemmas.

\begin{lemma} \label{lem:dichotomy}
For each $i = 1, \dots, k$, let $x_i$ be a real number between $0$ and $t$. Then 
$$\frac{1}{k} \sum_{i=1}^k (t - x_i)^k + \prod_{i=1}^k x_i \geq 2 (t/2)^k.$$
\end{lemma}

\begin{proof}
As the result is easily checked for $k = 2$, $3$ and $4$, we can assume without loss of generality that $k \geq 5$. Moreover, since $(t - t/k)^k > 2(t/2)^k k$ for all $k \geq 5$, we may assume that none of the $x_i$ are less than $t/k$.

We claim that the minimum value of $\sum_i (t - x_i)^k$ subject to the constraint $\prod_i x_i = z$, and assuming $x_i \geq t/k$ for all $i$, occurs when all the $x_i$ are equal to $z^{1/k}$. To see this, make the substitution $x_i = e^{y_i}$. The problem then becomes to minimise $\sum_{i=1}^k (t - e^{y_i})^k$ subject to the constraint $\sum_{i=1}^k y_i = \log z$. But the function $(t - e^y)^k$ is easily seen to be a convex function of $y$ for $t/k \leq e^y \leq t$. Therefore, the minimum occurs when all of the $e^{y_i}$ and, hence, all of the $x_i$ are equal.

Substituting $x_i = z^{1/k}$ for all $i$, it simply remains to minimise $f(z) = (t - z^{1/k})^k + z$ on the interval $[0, t^k]$. But $f'(z) = - (t - z^{1/k})^{k-1} z^{-(k-1)/k} + 1$, which equals $0$ precisely when $z = (t/2)^k$. Hence, the minimum value of $f(z)$ is $2 (t/2)^k$, as required.
\end{proof}

\begin{lemma} \label{lem:degprod}
Suppose that $k \leq \ell$ and, for each $i = 1, \dots, \ell$, let $x_i$ be a real number between $0$ and $1$. Then
\[\sum_{1\leq i_1 < \dots < i_k \leq \ell} \prod_{j=1}^k x_{i_j} \geq \binom{\sum_i x_i}{k}.\]
\end{lemma}

\begin{proof}
Suppose that $\sum_i x_i = c$ and we wish to minimise the left-hand side of the required inequality under this constraint. We claim that the minimum occurs when all but one of the $x_i$ equal $0$ or $1$, that is, $\lfloor c \rfloor$ of the $x_i$ are $1$, one is $\{c\} = c - \lfloor c \rfloor$ and the rest are $0$. 

Suppose instead that $x_1$ and $x_2$, say, are both different from $0$ and $1$. Then $x_1 x_2 = x_1 (c - \sum_{i=3}^k x_i - x_1)$, which has the form $-x_1^2 + Bx_1$, where $B$ is a function of $x_3, \dots, x_k$ and hence constant if these variables are held constant. But such a polynomial is minimised when $x_1$ is either as large or as small as possible within its allowed range. Hence, if $x_1$ and $x_2 = c - \sum_{i=3}^k x_i - x_1$ are both different from $0$ and $1$, we may vary $x_1$, keeping all $x_i$ with $3 \leq i \leq k$ fixed, to make $x_1 x_2$, and thus $\prod_{i=1}^k x_i$, smaller. This contradiction proves the claim, so
\[\sum_{1\leq i_1 < \dots < i_k \leq \ell} \prod_{j=1}^k x_{i_j} \geq \binom{\lfloor c \rfloor}{k} + \{ c\} \binom{\lfloor c \rfloor}{k-1} \geq \binom{c}{k}.\]
To establish the final inequality, suppose that $X$ is a random subset of a $(\lfloor c \rfloor + 1)$-element set, where the first element is chosen with probability $\{c\}$ and all other elements with probability $1$. The expected number of subsets of size $k$ in this random set is then 
$$\binom{\lfloor c \rfloor}{k} + \{ c\} \binom{\lfloor c \rfloor}{k-1}.$$ 
But it is also equal to 
$$(1 - \{c\}) \binom{\lfloor c \rfloor}{k} + \{c\}\binom{\lfloor c \rfloor + 1}{k},$$ 
which by convexity of $\binom{x}{k}$ is at least $\binom{c}{k}$.
\end{proof}

\section{Proof of Theorem~\ref{thm:main}}

Suppose that we have a red/blue-colouring of the edges of the complete graph on $N = (2^k + \epsilon) n$ vertices. Assume that $\eta$ is taken sufficiently small and $m_0$ sufficiently large in terms of $k$ and $\epsilon$ and apply Lemma~\ref{lem:main} with $\eta$ and $m_0$ to the red subgraph to obtain an equitable partition $\cup_{i=1}^m V_i$ of the vertex set $[N]$ with $m \geq m_0$ and subsets $W_i \subseteq V_i$ such that $W_i$ is $\eta$-regular for all $i$ and, for all but $\eta m^2$ pairs $(i, j)$ with $1 \leq i \neq j \leq m$, $(V_i, V_j)$, $(W_i, V_j)$ and $(W_i, W_j)$ are $\eta$-regular with $|d(W_i, V_j) - d(V_i, V_j)| \leq \eta$ and $|d(W_i, W_j) - d(V_i, V_j)| \leq \eta$, where $d(U, V)$ measures the red density between vertex sets $U$ and $V$. Because the colours are complementary, the same conclusion holds for the blue subgraph. For convenience of notation, we will assume below that all $V_i$ have precisely the same order $N/m$.

We now form a coloured reduced graph with vertex set $v_1, \dots, v_m$. To each $v_i$, we assign a colour $c_i$, either red or blue, depending on which colour has the higher density inside $W_i$, breaking ties arbitrarily. By the pigeonhole principle, at least $m/2$ of the $c_i$ are the same colour, say red. We now colour the edges of the reduced graph, leaving an edge uncoloured if $(W_i, V_j)$, $(V_i, V_j)$ and $(W_i, W_j)$ are not all $\eta$-regular with $|d(W_i, V_j) - d(V_i, V_j)| \leq \eta$ and $|d(W_i, W_j) - d(V_i, V_j)| \leq \eta$. Otherwise, we fix a constant $\delta$ (which will be taken sufficiently small in terms of $k$ and $\epsilon$) and colour the edge $v_i v_j$ red if the red density between $V_i$ and $V_j$ is at least $1 - \delta$ and blue if the blue density is at least $\delta$, again breaking ties arbitrarily. Note that there are at most $\eta m^2$ ordered pairs $(i, j)$ whose corresponding edge is uncoloured. Therefore, by deleting at most $\sqrt{\eta} m$ vertices, we may assume that each vertex is adjacent to at most $\sqrt{\eta} m$ uncoloured edges. In what follows, when referring to the reduced graph, we will assume that these vertices have been removed. Note that at least $s = \lceil (1/2 - \sqrt{\eta}) m \rceil$ of the remaining vertices have colour red.

Suppose now that there is a red vertex $v_a$ in the reduced graph which has degree at least $\ell := 2^{-k} m$ in red, with neighbours $v_{b_1}, \dots, v_{b_\ell}$. Since the density of red edges in $W_a$ is at least $1/2$, we may apply Lemma~\ref{lem:count2} with $U_1 = \dots = U_k = W_a$ and $U_{k+j} = V_{b_j}$ for $j = 1, \dots, \ell$ to conclude that, for $\eta$ sufficiently small in terms of $\delta$, there is a red $K_k$ which is contained in at least
\[\sum_{j=1}^\ell (d(W_a, V_{b_j})^k - \delta) |V_{b_j}| \geq ( (1 - \delta - \eta)^k - \delta) \ell \frac{N}{m} = ( (1 - \delta - \eta)^k - \delta) 2^{-k} N\]
red $K_{k+1}$. Provided $\eta$ and $\delta$ are sufficiently small in terms of $k$ and $\epsilon$, this quantity is at least $n$, so we obtain the required book $B_n^{(k)}$. We may therefore assume that we are in the other case, where every red vertex in the reduced graph has blue degree at least $m - \ell - 2\sqrt{\eta} m \geq (1 - 2^{-k} - 2\sqrt{\eta})m$.

The degree of each red vertex is therefore at least $(1 - 2^{-k} - 2\sqrt{\eta})m$ in blue. If we restrict to a set $S$ consisting of $s$ of the red vertices, the blue degree of each vertex inside this set is at least $s - (2^{-k} + 2\sqrt{\eta})m \geq (1 - 2^{-(k-1)} - 16 \sqrt{\eta})s$. Since $1 - 2^{-(k-1)} - 16 \sqrt{\eta} > 1 - (k-1)^{-1} + \beta$ for some $\beta > 0$ depending only on $k$ and the number of uncoloured edges is at most $\eta m^2 \leq 8 \eta s^2$, Lemma~\ref{lem:ES2} implies that for $m$ sufficiently large and $\eta$ sufficiently small in terms of $k$ and $t$, where $t$ is a constant to be fixed below, the reduced graph contains a blue blow-up of $K_k$ with at least $t$ vertices in each part, where each part is itself a monochromatic clique.

We now claim that none of these monochromatic cliques can be blue. Indeed, suppose otherwise and $C$ is a blue clique of order $t$. If any of the vertices in $C$, say $v_a$, is such that $\sum_j d(W_a, V_j) \geq \frac{1}{2} m$, where the sum is taken over all $j$ such that $(W_a, V_j)$ is $\eta$-regular, then we have
\[\sum_j d(W_a, V_j)^k \geq m \left(\frac{\sum_j d(W_a, V_j)}{m}\right)^k \geq 2^{-k} m.\]
Again, since the density of red edges in $W_a$ is at least $1/2$, we may apply Lemma~\ref{lem:count2} with $U_1 = \dots = U_k = W_a$ and $U_{k+j}$ equal in turn to each of the $V_j$ for which $(W_a, V_j)$ is $\eta$-regular to conclude that, for $\eta$ sufficiently small in terms of $\delta$, there is a red $K_k$ which is contained in at least
\[\sum_j (d(W_a, V_j)^k - \delta) |V_j| \geq (2^{-k} - \delta) N\]
red $K_{k+1}$. Provided $\eta$ and $\delta$ are sufficiently small in terms of $k$ and $\epsilon$, this quantity is at least $n$, so we again obtain the required book $B_n^{(k)}$.

Therefore, writing $\overline{d}(U, V)$ for the blue density between sets $U$ and $V$, we must have $\sum_j \overline{d}(W_a, V_j) \geq (\frac{1}{2} - 2\sqrt{\eta})m$ for all $v_a \in C$, where the sum is now over all $j$ such that $v_j$ is in the reduced graph. Writing $\overline{d}_C(V_j) = \sum_{v_a \in C} \overline{d}(W_a, V_j)$, we see, by applying Lemma~\ref{lem:degprod} and summing over all $j$ such that $v_j$ is in the reduced graph, that
\[\sum_j \sum_{(a_1, \dots, a_k) \in \binom{C}{k}} \prod_{i=1}^k \overline{d}(W_{a_i}, V_j) \geq \sum_j \binom{\overline{d}_C(V_j)}{k} \geq m \binom{\sum_j \overline{d}_C(V_j)/m}{k}.\]
Therefore, since $\sum_j \overline{d}_C(V_j) \geq \frac{1}{2} (1-4\sqrt{\eta})m |C|$, we have, for $t = |C| \geq (1 + \xi)k/(\xi - 4\sqrt{\eta})$, that
\[\sum_j \sum_{(a_1, \dots, a_k) \in \binom{C}{k}} \prod_{i=1}^k \overline{d}(W_{a_i}, V_j) \geq \binom{\frac{1}{2}(1 - 4\sqrt{\eta})|C|}{k} \geq \left(\frac{1}{2} (1 - \xi)\right)^{k} \binom{|C|}{k} \geq 2^{-k} (1 - k \xi) \binom{|C|}{k},\]
where we used that $\frac{1}{2} (1 - 4 \sqrt{\eta})|C| - i \geq \frac{1}{2} (1 - \xi) (|C| - i)$ for $0 \leq i \leq k$.
Hence, there exists a choice of $a_1, \dots, a_k$ such that
\[\sum_j \prod_{i=1}^k \overline{d}(W_{a_i}, V_j) \geq 2^{-k} (1 - k\xi) m.\]
Since, in the reduced graph, each $v_{a_i}$ has at most $\sqrt{\eta} m$ neighbours $v_j$ such that $(W_{a_i}, V_j)$ is not $\eta$-regular, if we now sum only over those $j$ such that $(W_{a_i}, V_j)$ is $\eta$-regular for all $i$, we have that
\[\sum_j \prod_{i=1}^k \overline{d}(W_{a_i}, V_j) \geq 2^{-k} (1 - k\xi) m - k \sqrt{\eta} m.\]
We now apply Lemma~\ref{lem:count2} with $U_i = W_{a_i}$ for each $1 \leq i \leq k$ and $U_{k+j}$ equal in turn to each $V_j$ with $(W_{a_i}, V_j)$ $\eta$-regular for all $1 \leq i \leq k$ to conclude that, for $\eta$ sufficiently small in terms of $\delta$, there is a blue $K_k$ which is contained in at least
\[\sum_j (\prod_{i=1}^k \overline{d}(W_{a_i}, V_j) - \delta) |V_j| \geq (2^{-k} (1 - k\xi) - k \sqrt{\eta} - \delta)N\]
blue $K_{k+1}$. Provided $\eta$, $\delta$ and $\xi$ are sufficiently small (and $t$ is sufficiently large) in terms of $k$ and $\epsilon$, this quantity is again at least $n$. 

This completes the proof of the claim. We may therefore assume that all of the cliques are red and focus on the subgraph of the reduced graph consisting of the $k$ red cliques $C_1, \dots, C_k$, each of order $t$, where every edge between $C_i$ and $C_j$ with $i \neq j$ is blue.

Now, for each vertex $v$ in the reduced graph, let $e_i(v)$ be the weighted blue degree of $v$ in each $C_i$. That is, $e_i(v) = \sum_{w \in C_i} \overline{d}(v, w)$. By Lemma~\ref{lem:dichotomy}, $\frac{1}{k} \sum_v \sum_i (t - e_i(v))^k + \sum_v \prod_i e_i(v) \geq 2(t/2)^k m'$, which implies that either $\sum_v \sum_i (t - e_i(v))^k \geq (t/2)^k k m'$ or $\sum_v \prod_i e_i(v) \geq (t/2)^k m'$, where $m' = (1 - \sqrt{\eta})m$. In the second case, we see that there must exist a choice of vertices $v_{c_1}, \dots, v_{c_k}$ with $v_{c_i} \in C_i$ such that 
\[\sum_j \prod_{i=1}^k \overline{d}(W_{c_i}, V_j) \geq \frac{\sum_j \sum_{c_1, \dots, c_k} \prod_i \overline{d}(W_{c_i}, V_j)}{t^k} = \frac{\sum_j \prod_i (\sum_{c_i \in C_i} \overline{d}(W_{c_i}, V_j))}{t^k} = \frac{\sum_v \prod_i e_i(v)}{t^k} \geq 2^{-k} m'.\]
Since there are at most $k \sqrt{\eta}m$ vertices $v_j$ such that $(W_{c_i}, V_j)$ is not $\eta$-regular for all $1 \leq i \leq k$, we may apply Lemma~\ref{lem:count2} with $U_i = W_{c_i}$ for $i = 1, \dots, k$ and $U_{k+j}$ equal in turn to each $V_j$ such that $(W_{c_i}, V_j)$ is $\eta$-regular for each $1 \leq i \leq k$ to conclude that, for $\eta$ sufficiently small in terms of $\delta$, there is a blue $K_k$ which is contained in at least
\[\sum_j (\prod_{i=1}^k \overline{d}(W_{c_i}, V_j) - \delta) |V_j| \geq (2^{-k}(1 - \sqrt{\eta}) - k \sqrt{\eta} - \delta)N\]
blue $K_{k+1}$, again giving the required book for $\eta$ and $\delta$ sufficiently small in terms of $k$ and $\epsilon$.

In the first case, there exists a $C_r$ such that $\sum_v (t - e_r(v))^k \geq (t/2)^k m'$. There must therefore exist (not necessarily distinct) vertices $d_1, \dots, d_k \in C_r$ such that 
\[\sum_j \prod_{i=1}^k d(W_{d_i}, V_j) \geq \frac{\sum_j \sum_{d_1, \dots, d_k} \prod_i d(W_{d_i}, V_j)}{t^k} = \frac{\sum_j (\sum_{d \in C_r} d(W_{d}, V_j))^k}{t^k} = \frac{\sum_v (t -e_r(v))^k}{t^k} \geq 2^{-k}m'.\]
If we again remove the at most $k \sqrt{\eta}m$ vertices $v_j$ such that $(W_{d_i}, V_j)$ is not $\eta$-regular for all $1 \leq i \leq k$, we may apply Lemma~\ref{lem:count2} with $U_i = W_{d_i}$ for $i = 1, \dots, k$ and $U_{k+j}$ equal in turn to each $V_j$ such that $(W_{d_i}, V_j)$ is $\eta$-regular for each $1 \leq i \leq k$ to conclude that, for $\eta$ sufficiently small in terms of $\delta$, there is a red $K_k$ which is contained in at least
\[\sum_j (\prod_{i=1}^k d(W_{d_i}, V_j) - \delta) |V_j| \geq (2^{-k}(1 - \sqrt{\eta}) - k \sqrt{\eta} - \delta)N\]
red $K_{k+1}$, giving the required book in this final case provided $\eta$ and $\delta$ are again small enough in terms of $k$ and $\epsilon$. This completes the proof.

\section{Concluding remarks}

One obvious question is whether a multicolour analogue of Theorem~\ref{thm:main} might hold. This is certainly not the case when the number of colours is large. To see this, we use the fact that there exist $q$-colourings of the complete graph on vertex set $\{1, 2, \dots, 2^{qk/4}\}$ with no monochromatic $K_k$ (see, for example, \cite[Section 2.1]{CFS15}). Fix such a colouring $\chi$. We consider the $(q+1)$-coloured complete graph whose vertex set is split into $t = 2^{qk/4}$ vertex sets $V_1, \dots, V_t$, each of order $n$, where every edge between $V_i$ and $V_j$ receives the colour $\chi(i, j)$ and edges internal to any $V_i$ all receive a $(q+1)$st colour. This colouring contains no monochromatic $B_n^{(k)}$, so the $(q+1)$-colour Ramsey number $r(B_n^{(k)}; q + 1) \geq 2^{qk/4} n$, far greater than the $(q+1)^k n$ bound one might hope for. More generally, we have $r(B_n^{(k)}; q + 1) \geq (r(k; q) - 1)n$, so, if true, the problem of showing that $r(B_n^{(k)}; 3) \leq 3^k n + o_k(n)$ is at least as hard as showing that $r(k) \leq 3^{k + o(k)}$. 

It is also tempting to generalise Theorem~\ref{thm:main} to hypergraphs. To this end, we define $B_n^{(k, s)}$ to be the $s$-uniform hypergraph consisting of $n$ copies of $K_{k+1}^{(s)}$, all sharing a common $K_k^{(s)}$. The natural conjecture would then be that
$$r(B_n^{(k,s)}) =  2^{\binom{k}{s-1}} n + o_{k,s}(n).$$
However, this is false for $s \geq 4$. To see this, suppose that $s \geq 3$, $k$ is a multiple of $s$ and there is a $2$-colouring $\chi$ of the $s$-uniform hypergraph on vertex set $\{1, 2, \dots, r - 1\}$ with no monochromatic $K_{k/s}^{(s)}$. Consider the complete $s$-uniform hypergraph whose vertex set is split into $r-1$ vertex sets $V_1, \dots, V_{r-1}$, each of order $n$. To colour this hypergraph, suppose that $\{v_1, \dots, v_s\}$ is an edge and $v_j \in V_{i_j}$ for all $1 \leq j \leq s$. If the $i_j$ are all distinct, we colour the edge by $\chi(i_1, \dots, i_s)$ and if the $i_j$ are all the same, we colour the edge red. Otherwise, we colour the edge blue. Since the colouring $\chi$ contains no monochromatic $K_{k/s}^{(s)}$, at least $s$ elements of the spine of any monochromatic $B_n^{(k,s)}$ are contained in the same set $V_i$. But this implies that the book must be red and, therefore, entirely contained within $V_i$, which is not large enough to contain it, a contradiction. Since we may take $r$ to be $r(K_{k/s}^{(s)})$, this implies that
$$r(B_n^{(k,s)}) \geq (r(K_{k/s}^{(s)}) - 1) n.$$ 
The value of $r(K_k^{(s)})$ is known to be at least an $(s-2)$-fold exponential in $k$ (see, for example,~\cite{CFS15}), so this disproves the conjecture for $s \geq 4$. The $s = 3$ case remains unresolved, though a negative answer would again follow from improved lower bounds for $r(K_k^{(3)})$.

As a final remark, we note that there is a strong analogy between Theorem~\ref{thm:main} and Green's popular progression theorem~\cite{G05} (see also~\cite{GT10}). This says that for every $\epsilon > 0$ there exists $n_0$ such that if $n \geq n_0$ and $A$ is a subset of $\{1, 2, \dots, n\}$ of size $\alpha n$, then there is $d \neq 0$ such that $A$ contains at least $(\alpha^3 - \epsilon)n$ arithmetic progressions of length $3$ with common difference $d$. That is, there are asymptotically at least as many arithmetic progressions of length $3$ in $A$ with common difference $d$ as there would be in a random subset of $\{1, 2, \dots, n\}$ of the same size. A surprising recent result of Fox, Pham and Zhao~\cite{FP18, FPZ18} says that $n_0$ grows as a tower-type function of $\epsilon$, showing that an application of the (arithmetic) regularity lemma in Green's proof is in some sense necessary. It would be very interesting if a similar phenomenon held for our result, though this seems unlikely to the author.

%%% AUTHOR: optional acknowledgments here
\section*{Acknowledgments} %%  you may comment this out if no Ackno
This paper was partially written while I was visiting the California Institute of Technology as a Moore Distinguished Scholar and I am extremely grateful for their kind support. I am also much indebted to Jacob Fox, Lisa Sauermann and Yuval Wigderson for pointing out a subtle error in the first version of this paper. Finally, I would like to thank the anonymous reviewers for several helpful remarks which improved the presentation.

%%% AUTHOR:
%%% Bibliography goes here. Note that the arXiv cannot process bibtex
%%% or biber bibliographies.  Example of acceptable bibliograpy format:
\bibliographystyle{amsplain}

%% AUTHOR: You can generate such a bibliography from a .bib file by 
%% running pdflatex/bibtex/pdflatex/pdflatex and then pasting the .bbl file
%% between \begin{thebibliography} and \end{bibliography}

%%% AUTHOR: Include a short description of each author following the
%%% structure below. Use the same short tags used previously.  
%%% Use \imageat{} and \imagedot{} instead of "@" and "." in
%%% email addresses-this replaces the symbols with graphics to avoid 
%%% e-mail address harvesting from the .pdf file
\begin{aicauthors}
\begin{authorinfo}[dcon]
  David Conlon\\
  Mathematical Institute\\
  University of Oxford\\
  david.conlon\imageat{}maths\imagedot{}ox\imagedot{}ac\imagedot{}uk \\
  \url{https://people.maths.ox.ac.uk/conlond}
\end{authorinfo}
\end{aicauthors}

\end{document}